\documentclass[12pt,reqno]{amsart}

\usepackage{times}
\usepackage[T1]{fontenc}
\usepackage{mathrsfs}
\usepackage{latexsym}
\usepackage[dvips]{graphics}
\usepackage{epsfig}
\usepackage{amsmath,amsfonts,amsthm,amssymb,amscd}
\input amssym.def
\input amssym.tex
\usepackage{color}
\usepackage{hyperref}
\usepackage{pdfpages}
\usepackage{amsmath,amsfonts,amsthm,amssymb,amscd}
\usepackage{pstricks}
\usepackage[myheadings]{fullpage}
\usepackage{caption}
\usepackage{subcaption}

\usepackage{mathdots}

\newcommand{\burl}[1]{\textcolor{blue}{\url{#1}}}

\newcommand\be{\begin{equation}}
\newcommand\ee{\end{equation}}
\newcommand\bea{\begin{eqnarray}}
\newcommand\eea{\end{eqnarray}}
\newcommand\bi{\begin{itemize}}
\newcommand\ei{\end{itemize}}
\newcommand\ben{\begin{enumerate}}
\newcommand\een{\end{enumerate}}

\newtheorem{theorem}{Theorem}[section]

\newtheorem{lemma}[theorem]{Lemma}

\newcommand{\twocase}[5]{#1 \begin{cases} #2 & \text{#3}\\ #4
&\text{#5} \end{cases}   }

\newcommand{\rs}{{\rm RS}}
\newcommand{\ra}{{\rm RA}}
\newcommand{\rso}{{\rm RS_{\rm obs}}}
\newcommand{\rao}{{\rm RA_{\rm obs}}}

\numberwithin{equation}{section}

\begin{document}

\title{Relieving and Readjusting Pythagoras}

\author{Victor Luo}\email{\textcolor{blue}{\href{mailto:victor.d.luo@williams.edu}{victor.d.luo@williams.edu}}}
\address{Department of Mathematics and Statistics, Williams College, Williamstown, MA 01267}

\author{Steven J. Miller}\email{\textcolor{blue}{\href{mailto:sjm1@williams.edu}{sjm1@williams.edu}},  \textcolor{blue}{\href{Steven.Miller.MC.96@aya.yale.edu}{Steven.Miller.MC.96@aya.yale.edu}}}
\address{Department of Mathematics and Statistics, Williams College, Williamstown, MA 01267}





\date{\today}

\thanks{The second named author was partially supported by NSF grant DMS0970067. We thank Kevin Dayaratna, Bernhard Klingenberg and Jeffrey Miller for helpful comments and code over the years.}

\begin{abstract}
Bill James invented the Pythagorean expectation in the late 70's to predict a baseball team's winning percentage knowing just their runs scored and allowed. His original formula estimates a winning percentage of $\rs^2/(\rs^2+\ra^2)$, where $\rs$ stands for runs scored and $\ra$ for runs allowed; later versions found better agreement with data by replacing the exponent 2 with numbers near 1.83. Miller and his colleagues provided a theoretical justification by modeling runs scored and allowed by independent Weibull distributions. They showed that a single Weibull distribution did a very good job of describing runs scored and allowed, and led to a predicted won-loss percentage of $(\rso-1/2)^\gamma / ((\rso-1/2)^\gamma + (\rao-1/2)^\gamma)$, where $\rso$ and $\rao$ are the observed runs scored and allowed and $\gamma$ is the shape parameter of the Weibull (typically close to 1.8). We show a linear combination of Weibulls more accurately determines a team's run production and increases the prediction accuracy of a team's winning percentage by an average of about 25\% (thus while the currently used variants of the original predictor are accurate to about four games a season, the new combination is accurate to about three). The new formula is more involved computationally; however, it can be easily computed on a laptop in a matter of minutes from publicly available season data. It performs as well (or slightly better) than the related Pythagorean formulas in use, and has the additional advantage of having a theoretical justification for its parameter values (and not just an optimization of parameters to minimize prediction error).
\end{abstract}

\maketitle

\tableofcontents

\section{Introduction}

The Pythagorean Win/Loss Formula, also known as the Pythagorean formula or Pythagorean expectation, was invented by Bill James in the late 1970s to use a team's observed runs scored and allowed to predict their winning percentage. Originally given by \be {\rm Won-Loss\ Percentage} \ = \ \frac{\rs^2}{\rs^2+\ra^2}, \ee with $\rs$ the runs scored and $\ra$ the runs allowed, it earned its name from the similarity of the denominator to the sums of squares in the Pythagorean formula from geometry.\footnote{Though of course the more natural shape in baseball is the diamond, save for some interesting stadium features, such as the triangle in Fenway Park.} Later versions found better agreement by replacing the exponent 2 with numbers near 1.83, leading to an average error of about three to four games per season.

The formula is remarkably simple, requiring only the runs scored and allowed by a team in a season, and the calculation (even with the improved exponent) is easily done on any calculator or phone. It is one of the most commonly listed expanded statistics on websites. One reason for its prominence is its accuracy in predicting a team's future performance through a simple calculation and not through computationally intense simulations. Additionally, it allows sabermetricians and fans to assess a manager's impact on a team, and estimate the value of new signings by seeing how their presence would change the predictions.

Because of its widespread use and utility, it is very desirable to have improvements. In his senior thesis, the first named author, supervised by the second named author, explored various attempted improvements to the Pythagorean formula. These included replacing the observed runs scored and allowed each game with adjusted numbers, with the adjustments coming from a variety of sources (such as ballpark effects, game state\footnote{For example, if a team is up by a large amount late in a game, they frequently use weaker relief pitchers and rest some starters, while the trailing team makes similar moves; thus the offensive productions from this point onward may not be indicative of the team's true abilities and a case can be made to ignore such data.}, WHIP, ERA+, and WAR of the pitcher, ...).

As these led to only minor improvements\footnote{The only adjusted formula that was at least on par or very near the accuracy of the original Pythagorean W/L Formula was that of ballpark factor.} (see \cite{Luo} for a detailed analysis of these and other adjustments), we turned our attention to the successful theoretical model used by Miller and his colleagues \cite{DaMil1,DaMil2,Mil,MCGLP}, where it was assumed runs scored and allowed were independently drawn from Weibull distributions with the same shape parameter. Recall the three parameter Weibull density is given by \be \twocase{f(x;\alpha,\beta,\gamma) \ = \ }{\frac{\gamma}{\alpha}\ ((x-\beta)/\alpha)^{\gamma-1}\ e^{- ((x-\beta)/\alpha)^{\gamma}}}{if $x \ge \beta$}{0}{otherwise.} \ee The effect of $\alpha$ is to control the spread of the output, while $\beta$ translates the distribution. The most important parameter is $\gamma$, which controls the shape. See Figure \ref{fig:weibullplots} for some plots.

\begin{center}
\begin{figure}
\includegraphics[scale=0.65]{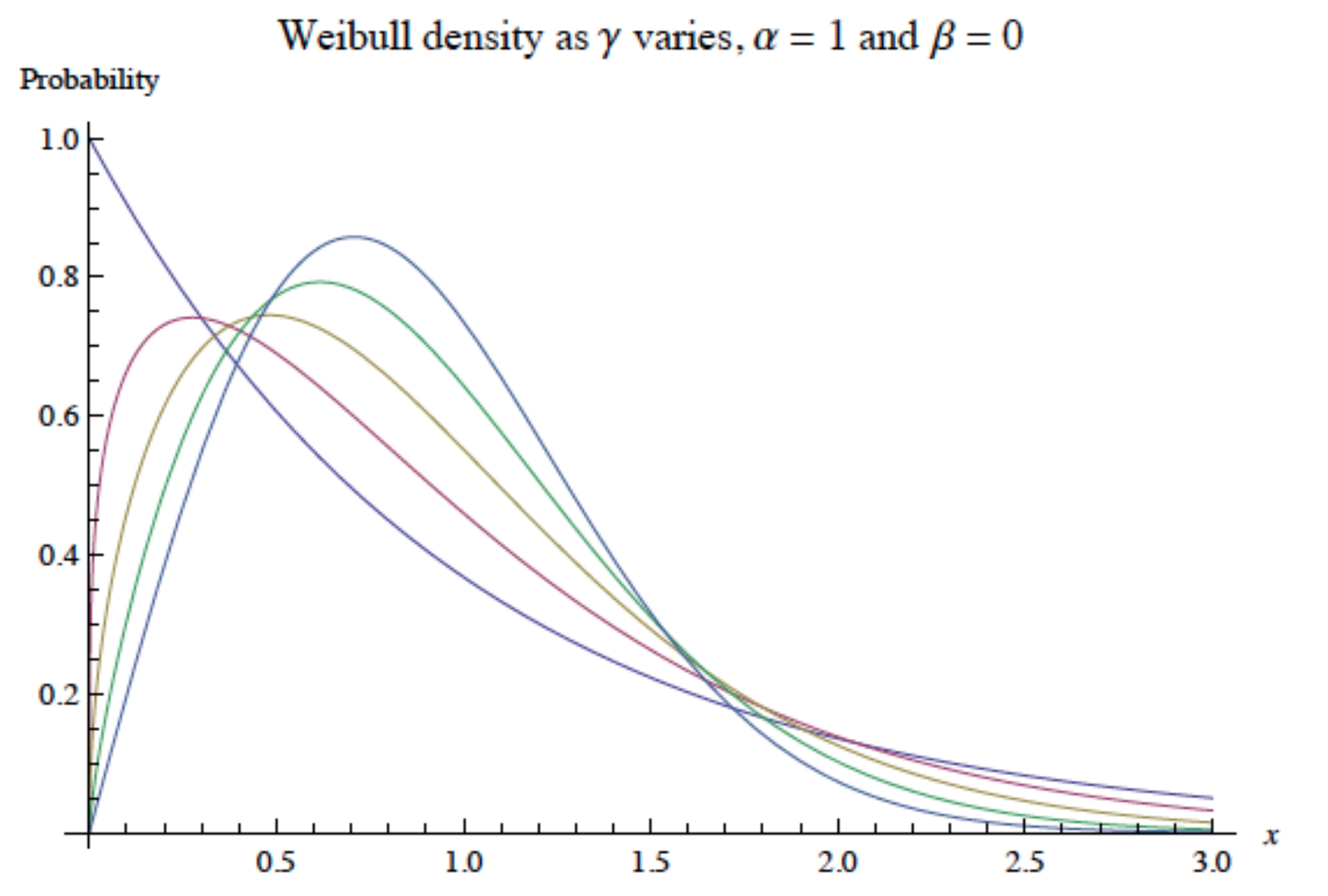}
\caption{\label{fig:weibullplots} The varying distributions of the Weibull family with $\alpha=1$ and $\beta=0$.}
\end{figure}
\end{center}

Their success is due to the fact that the three parameter Weibull is a very flexible family of distributions, capable of fitting many one hump distributions, including to a statistically significant degree the observed runs scored and allowed data. Miller chose to use Weibulls for two reasons. First, they lead to double integrals for the probabilities that can be evaluated in closed form. This is extremely important if we desire a simple expression such as the one posited by James (see \cite{HJM} for alternative simple formulas). Second, in addition to being flexible, special values of the Weibulls correspond to well-known distributions ($\gamma = 1$ is an exponential, while $\gamma = 2$ is the Rayleigh distribution).

The goal of this paper is to show that one can significantly improve the predictive power if instead of modeling runs scored and allowed as being drawn from independent Weibulls, we instead model them as being drawn from linear combinations of independent Weibulls. The advantage of this approach is that we are still able to obtain tractable double integrals which can be done in closed form. There is a cost, however, as now more analysis is needed to find the parameters and the correct linear combinations. While this results in a more complicated formula than the standard variant of James' formula, it is well worth the cost as on average it is better by one game per season (thus a typical error is 3 games per team per year, as opposed to 4 which is the typical result in the current formula). Comparing it to \url{baseball-reference.com}'s calculated Expected WL from 1979 to 2013, we find that the linear combination of Weibulls is approximately .06 of a game better, a difference is not statistically significant. However, there are noticeable trends that appear in certain eras.

\section{Theoretical Calculations} \label{sec:theoretical}

\subsection{Preliminaries}

It is important to note that we assume that runs scored and allowed are taken from continuous, not discrete, distributions. This allows us to deal with continuous integrals rather than discrete sums, which most of the time leads to easier calculations. While a discrete distribution would probably more effectively map runs in baseball, the assumption of drawing runs from a continuous distribution allows for more manageable calculations, and is a very sensible estimate of those runs observed. It also will lead to closed form expressions, which are much easier to work with and allow us to avoid having to resort to simulations.

The Weibulls lead to significantly easier calculations because if we have a random variable $X$ chosen from a Weibull distribution with parameters $\alpha$, $\beta$, and $\gamma$, then $X^{1/\gamma}$ is exponentially distributed with parameter $\alpha^\gamma$; thus, a change of variables yields a simpler integral of exponentials, which can be done in closed form (see Appendix 9.1 in \cite{MCGLP} for details).

In all arguments below we always take $\beta=-1/2$, though we often write $\beta$ to keep the discussion more general for applications to other sports. The reason we do this is that we use procedures such as the Method of Least Squares to find the best fit parameters, and this requires binning the observed runs scored and allowed data. As baseball scores are discrete, there are issues if these values occur at the boundary of bins; it is much better if they are at the center. By taking $\beta=-1/2$ we break the data into bins
\begin{equation} \label{bins}
\left[-\frac{1}{2},\ \frac{1}{2}\right),\ \ \ \left[\frac{1}{2},\ \frac{3}{2}\right),\ \ \ \left[\frac{3}{2},\ \frac{5}{2}\right),\ \ \ \cdots.
\end{equation}

Our final assumption is that runs scored and runs allowed are independent. This obviously cannot be true, as a baseball game never ends in a tie. For example, if the Orioles and the Red Sox are playing and the O's score 5 runs, then the Sox cannot score 5. Statistical analyses support this hypothesis; see the independence tests with structural zeros in \cite{Mil} or Appendix 9.2 in \cite{MCGLP} for details.

We end this subsection with the mean and the variance of the Weibull. The calculation follows from standard integration (see the expanded version of \cite{Mil} or Appendix 9.1 in \cite{MCGLP} for a proof of the formula for the mean; the derivation of the variance follows in a similar fashion).

\begin{lemma} \label{weibullmv} Consider a Weibull with parameters $\alpha, \beta, \gamma$. The mean, $\mu_{\alpha,\beta,\gamma}$, equals
\be
\mu_{\alpha,\beta,\gamma}\ = \ \alpha \Gamma (1+\gamma^{-1})+\beta \label{weibullmean}
\ee
while the variance, $\sigma^2_{\alpha,\beta,\gamma}$, is
\be
\sigma^2_{\alpha,\beta,\gamma}\ = \ \alpha^2 \Gamma(1+2\gamma^{-1})-\alpha^2\Gamma(1+\gamma^{-1})^2,
\ee
where $\Gamma(x)$ is the Gamma function, defined by $\Gamma(x)=\int_0^\infty e^{-u}u^{x-1}du$.
\end{lemma}

\subsection{Linear Combination of Weibulls}

We now state and prove our main result for a linear combination of two Weibulls, and leave the straightforward generalization to combinations of more Weibulls to the reader. The reason such an expansion is advantageous and natural is that, following \cite{Mil}, we can integrate pairs of Weibulls in the regions needed and obtain simple closed form expressions. The theorem below also holds if $\gamma < 0$; however, in that situation the more your runs scored exceeds your runs allowed, the worse your predicted record due to the different shape of the Weibull (in all applications of Weibulls in survival analysis, the shape parameter $\gamma$ must be positive).

\begin{theorem} \label{combweib}
Let the runs scored and allowed per game be two independent random variables drawn from linear combinations of independent Weibull distributions with the same $\beta$'s and $\gamma$'s. Specifically, if $W(t;\alpha,\beta,\gamma)$ represents a Weibull distribution with parameters $(\alpha,\beta,\gamma)$,  and we choose non-negative weights\footnote{If we had more terms in the linear combination, we would simply choose non-negative weights summing to 1.} $0\le c_i, c_j' \le 1$  (so $c_1 + c_2 = 1$ and $c_1'+c_2' = 1$), then the density of runs scored, $X$ is \be f(x;\alpha_{{\rm RS}_1},\alpha_{{\rm RS}_2}, \beta, \gamma,c_1, c_2) \ = \ c_1 W(x;\alpha_{{\rm RS}_1},\beta,\gamma)+c_2 W(\alpha_{{\rm RS}_2},\beta,\gamma)\ee and runs allowed, $Y$, is \be f(y;\alpha_{{\rm RA}_1},  \alpha_{{\rm RA}_2}, \beta, \gamma,c_1', c_2') \ = \ c_1' W(y;\alpha_{{\rm RA}_1},\beta,\gamma)+c_2'W(\alpha_{{\rm RA}_2},\beta,\gamma).\ee In addition, we choose $\alpha_{{\rm RS}_1}$ and $\alpha_{{\rm RS}_2}$ so that the mean of $X$ is $\rso$ and choose $\alpha_{{\rm RA}_1}$ and $\alpha_{{\rm RA}_2}$ such that the mean of $Y$ is $\rao$. For $\gamma>0$, we have
\bea & & {\rm Won-Loss\ Percentage}(\alpha_{{\rm RS}_1}, \alpha_{{\rm RS}_2}, \alpha_{{\rm RA}_1}, \alpha_{{\rm RA}_2},\beta,\gamma,c_1, c_2, c_1',c_2') \nonumber\\ & & \ \ \ \ \ \ \ \ \ \ \ \ \ \ \  =\  c_1 c_1' \frac{\alpha_{{\rm RS}_1}^\gamma}{\alpha_{{\rm RS}_1}^\gamma+\alpha_{{\rm RA}_1}^\gamma} +c_1 c_2'\frac{\alpha_{{\rm RS}_1}^\gamma}{\alpha_{{\rm RS}_1}^\gamma+\alpha_{{\rm RA}_2}^\gamma} \nonumber\\ & & \ \ \ \ \ \ \ \ \ \ \ \ \ \ \  \ \ \ \ +c_2 c_1' \frac{\alpha_{{\rm RS}_2}^\gamma}{\alpha_{{\rm RS}_2}^\gamma +\alpha_{{\rm RA}_1}^\gamma} +c_2 c_2'\frac{\alpha_{{\rm RS}_2}^\gamma}{\alpha_{{\rm RS}_2}^\gamma+\alpha_{{\rm RA}_2}^\gamma} \nonumber\\ & & \ \ \ \ \ \ \ \ \ \ \ \ \ \ \  = \ \sum_{i=1}^2 \sum_{j=1}^2 c_i c_j' \frac{\alpha_{{\rm RS}_i}^\gamma}{\alpha_{{\rm RS}_i}^\gamma+\alpha_{{\rm RA}_j}^\gamma}. \eea
\end{theorem}

\begin{proof}
As the means of $X$ (runs scored) and $Y$ (runs allowed) are $\rso$ and $\rao$, respectively, and the random variables are drawn from linear combinations of independent Weibulls, by Lemma \ref{weibullmv}
\begin{align}
\rso &\ =\ c_1(\alpha_{{\rm RS}_1}\Gamma(1+\gamma^{-1})+\beta)+(1-c_1)(\alpha_{{\rm RS}_2}\Gamma(1+\gamma^{-1})+\beta) \nonumber \\
\rao &\ =\ c_1'(\alpha_{{\rm RA}_1}\Gamma(1+\gamma^{-1})+\beta)+(1-c_1')(\alpha_{{\rm RA}_2}\Gamma(1+\gamma^{-1})+\beta).
\end{align}

We now calculate the probability that $X$ exceeds $Y$. We constantly use the fact that the integral of a probability density is 1. We need the two $\beta$ and the two $\gamma$'s to be equal in order to obtain closed form expressions.\footnote{If the $\beta$'s are differen then in the integration below we might have issues with the bounds of integration, while if the $\gamma$'s are unequal we get incomplete Gamma functions, though for certain rational ratios of the $\gamma$'s these can be done in closed form.} We find
\begin{align}
& {\rm Prob}(X>Y)  = \int_{x=\beta}^\infty\int_{y=\beta}^x f(x; \alpha_{{\rm RS}_1},\alpha_{{\rm RS}_2},\beta,\gamma, c_1, c_2)f(y; \alpha_{{\rm RA}_1},\alpha_{{\rm RA}_2},\beta,\gamma, c_1',c_2')dydx  \nonumber \\ 
 &=\  \sum_{i=1}^2\sum_{j=1}^2 \int_{x=0}^\infty \int_{y=0}^x c_ic_j' \frac{\gamma}{\alpha_{{\rm RS}_i}}\left(\frac{x}{\alpha_{{\rm RS}_i}}\right)^{\gamma-1}e^{-(\frac{x}{\alpha_{{\rm RS}_i}})^\gamma}\frac{\gamma}{\alpha_{{\rm RA}_j}}\left(\frac{x}{\alpha_{{\rm RA}_j}}\right)^{\gamma-1}e^{-(\frac{x}{\alpha_{{\rm RA}_j}})^\gamma}dydx \nonumber \\  
& = \sum_{i=1}^2\sum_{j=1}^2 c_ic_j' \int_{x=0}^\infty \frac{\gamma}{\alpha_{{\rm RS}_i}} \left( \frac{x}{\alpha_{{\rm RS}_i}} \right)^{\gamma-1}e^{-(\frac{x}{\alpha_{{\rm RS}_i}})^\gamma}\left[ \int_{y=0}^x \frac{\gamma}{\alpha_{{\rm RA}_j}} \left( \frac{y}{\alpha_{{\rm RA}_j}} \right)^{\gamma-1} e^{-(\frac{y}{\alpha_{{\rm RA}_j}})^\gamma} dy \right] dx \nonumber \\
 &=\  \sum_{i=1}^2\sum_{j=1}^2 c_ic_j'  \int_{x=0}^\infty \frac{\gamma}{\alpha_{{\rm RS}_i}} \left( \frac{x}{\alpha_{{\rm RS}_i}} \right)^{\gamma-1}e^{-(\frac{x}{\alpha_{{\rm RS}_i}})^\gamma} \ast \left[ 1- e^{-(\frac{x}{\alpha_{{\rm RA}_j}})^\gamma} \right] dx \nonumber \\ 
 &=\  \sum_{i=1}^2\sum_{j=1}^2 c_ic_j'  \left[ 1- \int_{x=0}^\infty \frac{\gamma}{\alpha_{{\rm RS}_i}} \left( \frac{x}{\alpha_{{\rm RS}_i}}\right)^{\gamma-1} e^{-(\frac{x}{\alpha_{{\rm RS}_i}})^\gamma-(\frac{x}{\alpha_{{\rm RA}_j}})^\gamma} dx\right]. \label{eqqq}
\end{align}

We set
\begin{align}
\frac{1}{\alpha_{ij}^\gamma}\ = \ \frac{1}{\alpha_{{\rm RS}_i}^\gamma}+\frac{1}{\alpha_{{\rm RA}_j}^\gamma}\ = \ \frac{\alpha_{{\rm RS}_i}^\gamma+\alpha_{{\rm RA}_j}^\gamma}{\alpha_{{\rm RS}_i}^\gamma \alpha_{{\rm RA}_j}^\gamma} \nonumber \\ \nonumber
\end{align}
for $1\leq i,j\leq 2$, so that \eqref{eqqq} becomes
\begin{align}
& \sum_{i=1}^2\sum_{j=1}^2 c_ic_j'  \left[ 1- \int_{x=0}^\infty \frac{\gamma}{\alpha_{{\rm RS}_i}}\left( \frac{x}{\alpha_{{\rm RS}_i}}\right)^{\gamma-1}e^{-(\frac{x}{\alpha_{ij}})^\gamma} dx \right]  \nonumber \\
&=\ \sum_{i=1}^2\sum_{j=1}^2 c_ic_j' \left[ 1- \frac{\alpha_{ij}^\gamma}{\alpha_{{\rm RS}_i}^\gamma} \int_{x=0}^\infty \frac{\gamma}{\alpha_{ij}} \left( \frac{x}{\alpha_{ij}} \right)^{\gamma-1}e^{-(\frac{x}{\alpha_{ij}})^\gamma} dx \right] \nonumber \\
&=\  \sum_{i=1}^2 \sum_{j=1}^2 c_ic_j' \left[1-\frac{\alpha_{ij}^\gamma}{\alpha_{{\rm RS}_i}^\gamma} \right] \nonumber \\ 
&=\  \sum_{i=1}^2\sum_{j=1}^2 c_i c_j'\left[1-\frac{1}{\alpha_{{\rm RS}_i}^\gamma}\ast \frac{\alpha_{{\rm RS}_i}^\gamma\alpha_{{\rm RA}_j}^\gamma}{\alpha_{{\rm RS}_i}^\gamma+\alpha_{{\rm RA}_j}^\gamma}\right] \nonumber \\  
&=\  \sum_{i=1}^2\sum_{j=1}^2 c_i c_j'\left[\frac{\alpha_{{\rm RS}_i}^\gamma}{\alpha_{{\rm RS}_i}^\gamma+\alpha_{{\rm RA}_j}^\gamma}\right] \nonumber \\  
&=\  c_1c_1' \frac{\alpha_{{\rm RS}_1}^\gamma}{\alpha_{{\rm RS}_1}^\gamma+\alpha_{{\rm RA}_1}^\gamma} +c_1c_2' \frac{\alpha_{{\rm RS}_1}^\gamma}{\alpha_{{\rm RS}_1}^\gamma+\alpha_{{\rm RA}_2}^\gamma} \nonumber \\
&\indent +c_2c_1' \frac{\alpha_{{\rm RS}_2}^\gamma}{\alpha_{{\rm RS}_2}^\gamma +\alpha_{{\rm RA}_1}^\gamma}
+c_2c_2' \frac{\alpha_{{\rm RS}_2}^\gamma}{\alpha_{{\rm RS}_2}^\gamma+\alpha_{{\rm RA}_2}^\gamma},
\end{align} completing the proof of Theorem \ref{combweib}.\end{proof}


\section{Curve Fitting}

\subsection{Theory}

We now turn to finding the values of the parameters leading to the best fit. We require $\beta = -1/2$ (for binning purposes), but otherwise the parameters ($\alpha_{{\rm RS}_1}, \alpha_{{\rm RS}_2}$, $\alpha_{{\rm RA}_1}$, $\alpha_{{\rm RA}_2}$, $\gamma$, $c_1$, $c_2$, $c_1'$, $c_2'$) are free.\footnote{Subject to, of course, $0 \le c_i, c_j' \le 1$ and $c_1+c_1=c_1'+c_2'=1$.} Our first approach was to use the Method of Moments, where we compute the number of moments equal to the number of parameters. Unfortunately the resulting equations were too involved to permit simple solutions for them in terms of the observed data; for completeness they are given in Appendix \ref{sec:AppMoment} (or see \cite{Luo}). We thus turned to the Method of Least Squares (though one could also do an analysis through the Method of Maximum Likelihood).

We looked at the 30 teams of the entire league from the 2004 to 2012 season. We display results from the 2011, but the results from any other season are similar readily available (see \cite{Luo}). We implemented the Method of Least Squares using the bins in \eqref{bins},  which involved minimizing the sum of squares of the error of the runs scored data plus the sum of squares of the error of the runs allowed data. There were seven free parameters: $\alpha_{{\rm RS}_1}$, $\alpha_{{\rm RS}_2}$, $\alpha_{{\rm RA}_1}$, $\alpha_{{\rm RA}_2}$, $\gamma$, $c_1$, and $c_1'$. Letting Bin$(k)$ be the $k$\textsuperscript{th} bin of $\eqref{bins}$, $\rso(k)$ and $\rao(k)$ represent the observed number of games with number of runs scored and allowed in Bin$(k)$, and $A(\alpha_1,\alpha_2,\beta,\gamma,c_1,k)$ denote the area under the linear combination of two Weibulls with parameters $(\alpha_1,\alpha_2,\beta,\gamma,c_1)$ in Bin$(k)$, then for each team we found the values of $(\alpha_{{\rm RS}_1},\alpha_{{\rm RS}_2},\alpha_{{\rm RA}_1},\alpha_{{\rm RA}_2}, \gamma, c_1, c_1')$ that minimized
\begin{align}
& \sum_{k=1}^{\textnormal{Num. Bins}} (\rso(k)-\#\textnormal{Games}\ast A(\alpha_{{\rm RS}_1},\alpha_{{\rm RS}_2},-.5,\gamma,c_1,k))^2 \nonumber \\ &\indent +\sum_{k=1}^{\textnormal{Num. Bins}} (\rao(k)-\#\textnormal{Games}\ast A(\alpha_{{\rm RA}_1},\alpha_{{\rm RA}_2},-.5,\gamma,c_1',k))^2 .
\end{align}



\subsection{Results}

For each team, we found the best fit linear combination of Weibulls. In Figure \ref{leastsquarestable}, we compared the predicted wins, losses, and won-loss percentage with the observed ones.

\begin{figure}[h!]
\includegraphics[scale=0.5]{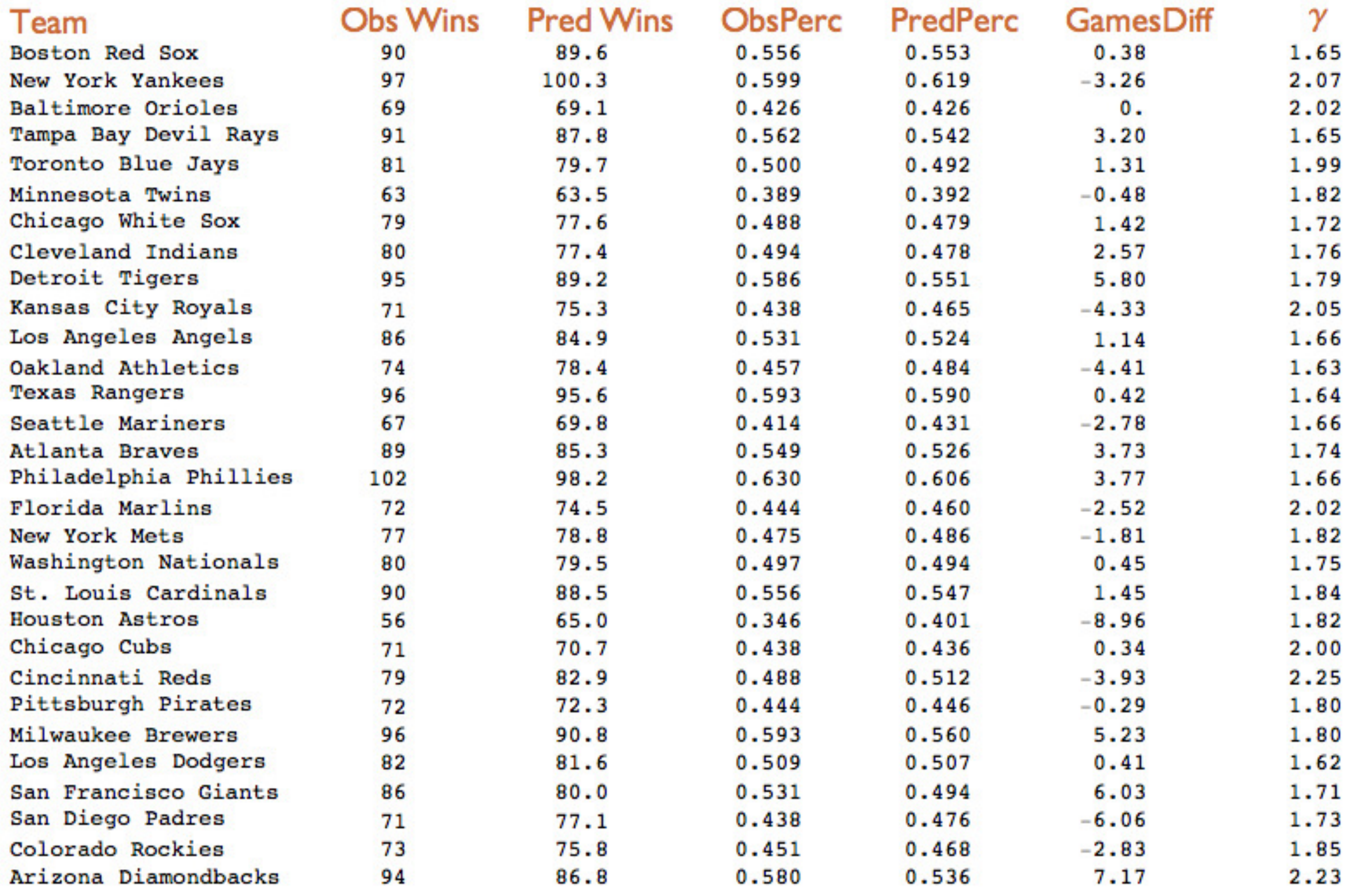}
\caption{Results for the 2011 season using Method of Least Squares. \label{leastsquarestable}}
\end{figure}

The code used is available in \cite{Luo}. Using the Method of Least Squares, the mean $\gamma$ over all 30 teams is 1.83 with a standard deviation of 0.18 (the median is 1.79). We can see that the exponent 1.83, considered as the best exponent, is clearly within the region of one standard deviation from the mean $\gamma$. Considering the absolute value of the difference between observed and predicted wins,  we have a mean of 2.89 with a standard deviation of 2.34 (median is 2.68). Without considering the absolute value, the mean is 0.104 with a standard deviation of 3.75 (and a median of 0.39). We only concern ourselves with the absolute value of the difference, as this really tells how accurate our predicted values are.

These values are significant improvements on those obtained when using a single Weibull distribution to predict runs (which essentially reproduces James' original formula, though  with a slightly different exponent), which produces a mean number of games off of 4.43 with standard deviation 3.23 (and median 3.54) in the absolute value case. We display the results over seasons from 2004 to 2012 in Figure \ref{sidebyside}. It is apparent that the linear combination of Weibulls better estimates teams' win/loss percentage; in fact, it is over one game better at estimating than the single Weibull! The mean number of games off for a single Weibull from 2004 to 2012 was 4.22 (with a standard deviation of 3.03), while that of the linear combination of Weibulls was 3.11 (with a standard deviation of 2.33). In addition, there is less standard deviation in the estimates. Thus, it appears that the linear combination of Weibulls provides a much tighter, better estimate than the single Weibull does.

\begin{figure}[h!]
\includegraphics[scale=0.4]{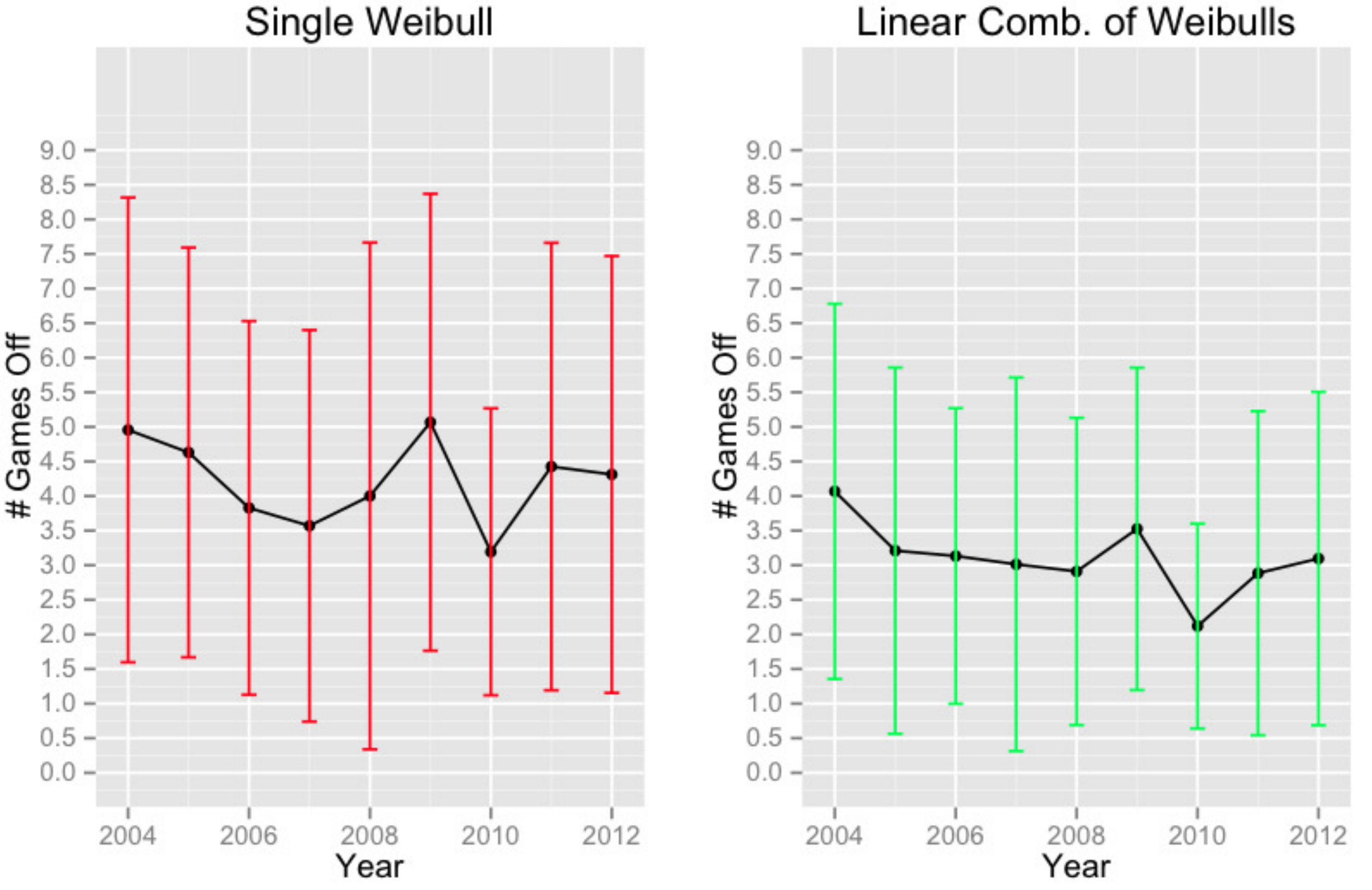}
\caption{Mean number of games off (with standard deviation) for single Weibull and linear combination of Weibulls from 2004-2012. \label{sidebyside}}
\end{figure}

To further demonstrate how accurate the quality of the fit is, we compare the best fit linear combination of Weibulls of runs scored and allowed with those observed of the 2011 Seattle Mariners in Figure \ref{mariners}; we can see that the fit is visually very good. Of course, the fit \emph{cannot} be worse, as we can always set $c_1 = 0=c_1'$; however, we can see the linear combination of Weibulls does a better job tracking the shape of the runs scored.

\begin{figure}
\centering
\begin{subfigure}{.53\textwidth}
  \centering
  \includegraphics[width=.8\linewidth]{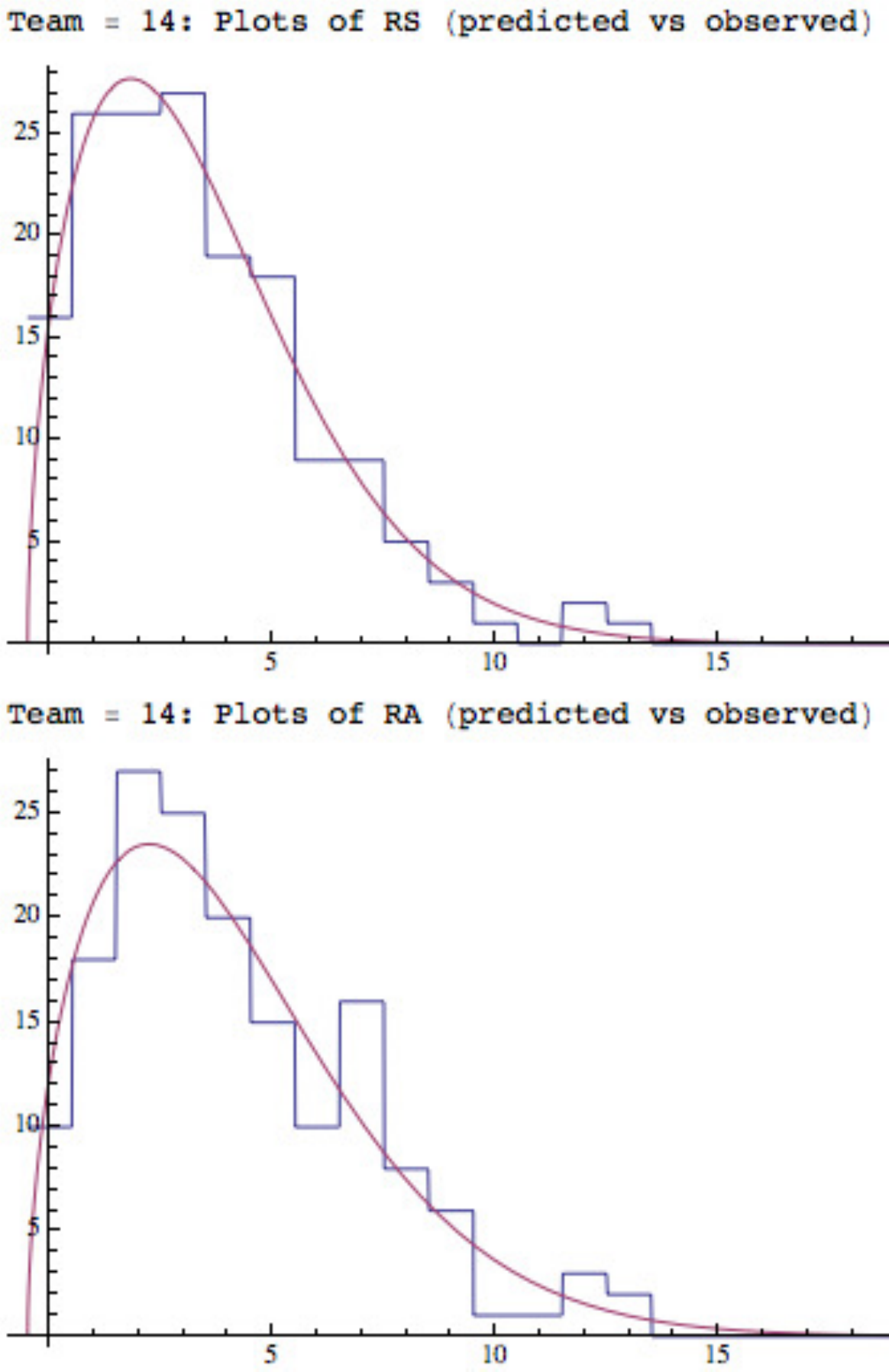}
  \caption{Single Weibull mapping runs scored and \\ allowed.}
  \label{fig:sub1}
\end{subfigure}%
\begin{subfigure}{.53\textwidth}
  \centering
  \includegraphics[width=.8\linewidth]{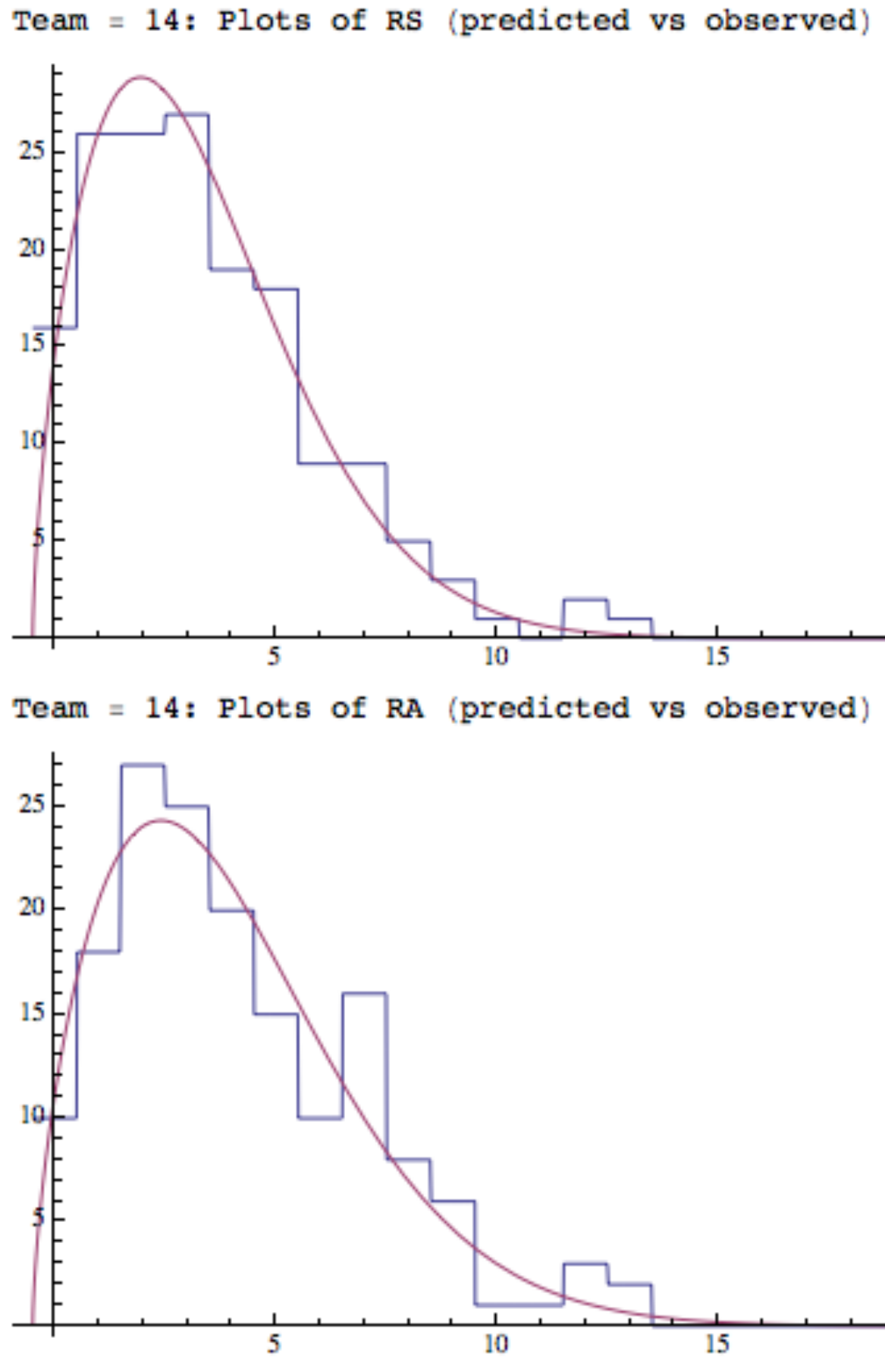}
  \caption{Linear Combination of Weibulls mapping runs \\ scored and allowed.}
  \label{fig:sub2}
\end{subfigure}
\caption{Comparison of best fit linear combination of Weibulls versus single Weibull for runs scored (top) and allowed (bottom) for the 2011 Seattle Mariners against the observed distribution of scores.}
\label{mariners}
\end{figure}

We then performed an independent two-sample t-test with unequal variances in $R$ using the t.test command to see if the difference between the games off determined by the single Weibull and those by linear combinations of Weibulls is statistically significant in Figure \ref{test}. With a $p$-value less than 0.01 and a 95\% confidence interval that does not contain 0, we can see that the difference is in fact statistically significant.

\begin{figure}[h!]
\includegraphics[scale=0.7]{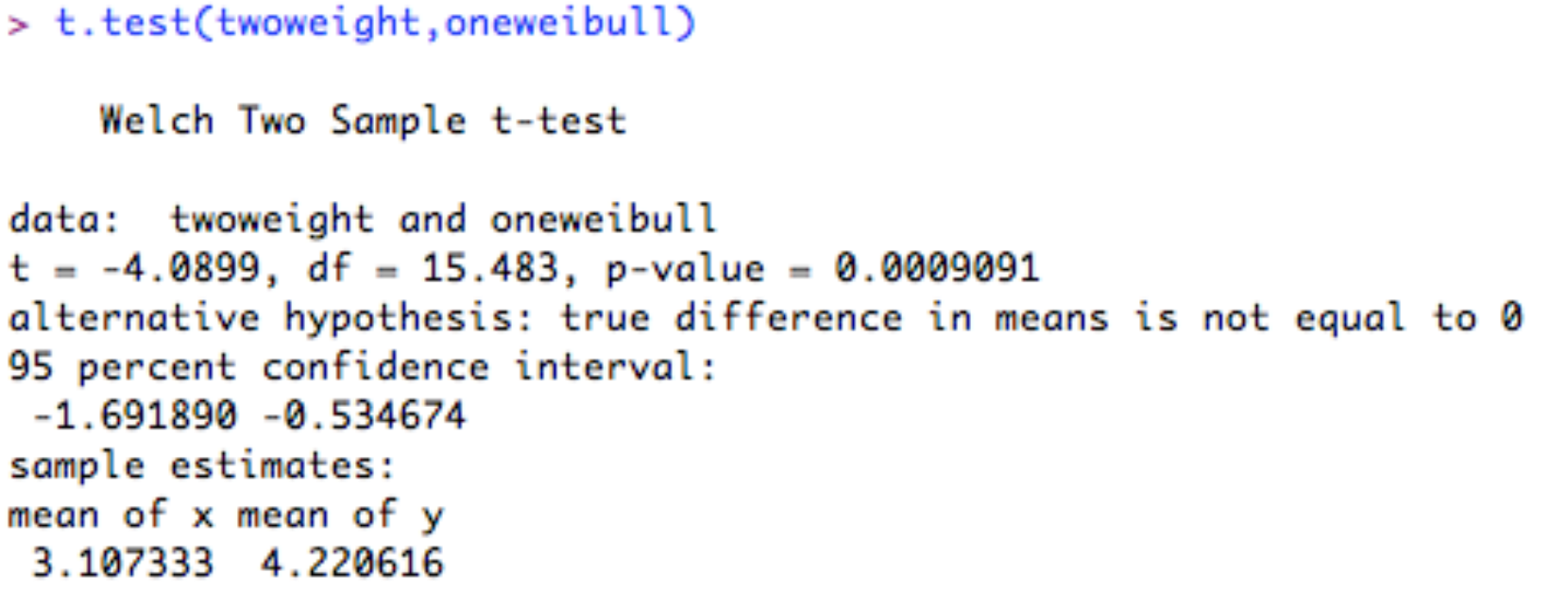}
\caption{t-test to determine whether the difference between the games off determined by the single Weibull and those by linear combinations of Weibulls is statistically significant.  \label{test}}
\end{figure}

In addition, we compared the mean number of games off of \url{baseball-reference.com}'s Pythagorean Win-Loss statistic (pythWL) and those of the linear combination of Weibulls from 1979 to 2013. Originally, we used ESPN's ExWL statistic\footnote{See the bottom of the page \url{http://espn.go.com/mlb/stats/rpi/_/year/2011}.} which used an exponent of 2; however, ESPN only went down to the year 2002, and it has been shown that using the exponent 1.83 is more accurate than using ESPN's exponent of 2. Using \url{baseball-reference.com}'s Pythagorean Win-Loss statistic (pythWL) to obtain more data (\url{baseball-reference.com} allowed us to go all the way down to 1979, rather than just 2002, which ESPN gives), the pythWL statistic\footnote{At \url{http://www.sports-reference.com/blog/baseball-reference-faqs/} see the section ``What is Pythagorean Winning Percentage?".} is calculated as
\begin{center}
$({\text{runs scored}^{1.83}})/({\text{runs scored}^{1.83}+\text{runs allowed}^{1.83}})$.
\end{center}

We display the results of our comparisons in Figure \ref{espnplot}. The mean number of games off for the pythWL statistic was 3.09 with a standard deviation of 2.26, numbers only slightly worse than those of the linear combination of Weibulls (mean of 3.03 with standard deviation of 2.21). So, we can see that the linear combination of Weibulls is doing, on average, about .06 of a game better than the pythWL statistic. We performed an independent two-sample t-test with unequal variances in $R$ using the t.test command to see if the difference between the games off determined by the pythWL statistic and the linear combinations of Weibulls is statistically significant in Figure \ref{ttest2}. With a very large $p$-value, we fail to reject the null hypothesis, suggesting that the difference in mean number of games off is not in fact statistically significant. We also display a plot (Figure \ref{difference}) that models the difference in the number of games between the pythWL statistic and the linear combination of Weibulls; it seems to be a constant positive value for the most part, suggesting that the linear combination of Weibulls is doing slightly better than the pythWL statistic.

\begin{figure}[h!]
\includegraphics[scale=0.46]{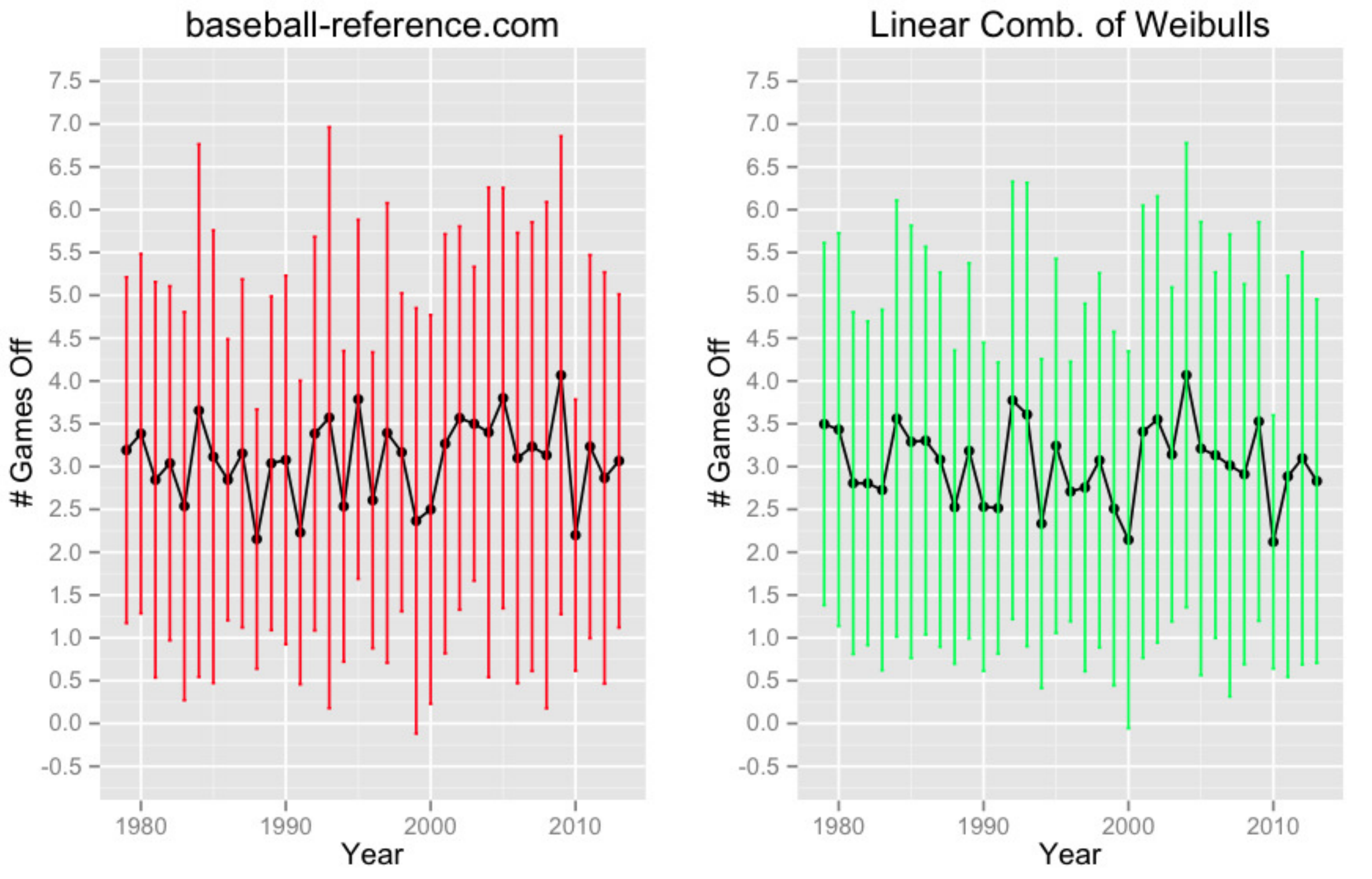}
\caption{Mean number of games off (with standard deviation) for \url{baseball-reference.com}'s pythWL statistic and linear combination of Weibulls from 1979-2013. \label{espnplot}}
\end{figure}

Looking at Figure \ref{difference} more closely, we can see that there are parts/eras of the graph in which the pythWL statistic does better, and parts where the linear combination of Weibulls does better. In the era from 1979-1989, the pythWL statistic is more accurate, beating the linear combination of Weibulls in 7 out of the 11 years. However, from 1990 to 2013, the linear combination of Weibulls wins in 15 out of the 24 years, and does so by around 0.3 games in those years. Furthermore, when the pythWL statistic does beat the linear combination of Weibulls in the years from 1990 to 2013, it does so by around 0.25 games, including the point at 2004, which seems very out of the ordinary; without this point, the pythWL statistic wins by about .2 games in the years between 1990 and 2013 that it does beat the linear combination of Weibulls. Thus, in more recent years, it may make more sense to use the linear combination of Weibulls. In addition, with respect to the standard deviation of number of games off of the pythWL statistic (2.26) and the linear combination of Weibulls (2.21), we can see that the linear combination of Weibulls provides on average a tighter fit, i.e., there is less fluctuation in the mean number of games off for each team in each year (from 1990 to 2013, the pythWL statistic standard deviation in games off is 2.34 while that of the linear combination of Weibulls is 2.22, so we again see that the linear combination does noticeably better in recent years).

It is important to note that the pythWL statistic just takes the functional form of the Pythagorean Win/Loss Formula with an exponent ($\gamma$) of 1.83, while we give theoretical justification for our formula.


\begin{figure}[h!]
\includegraphics[scale=0.6]{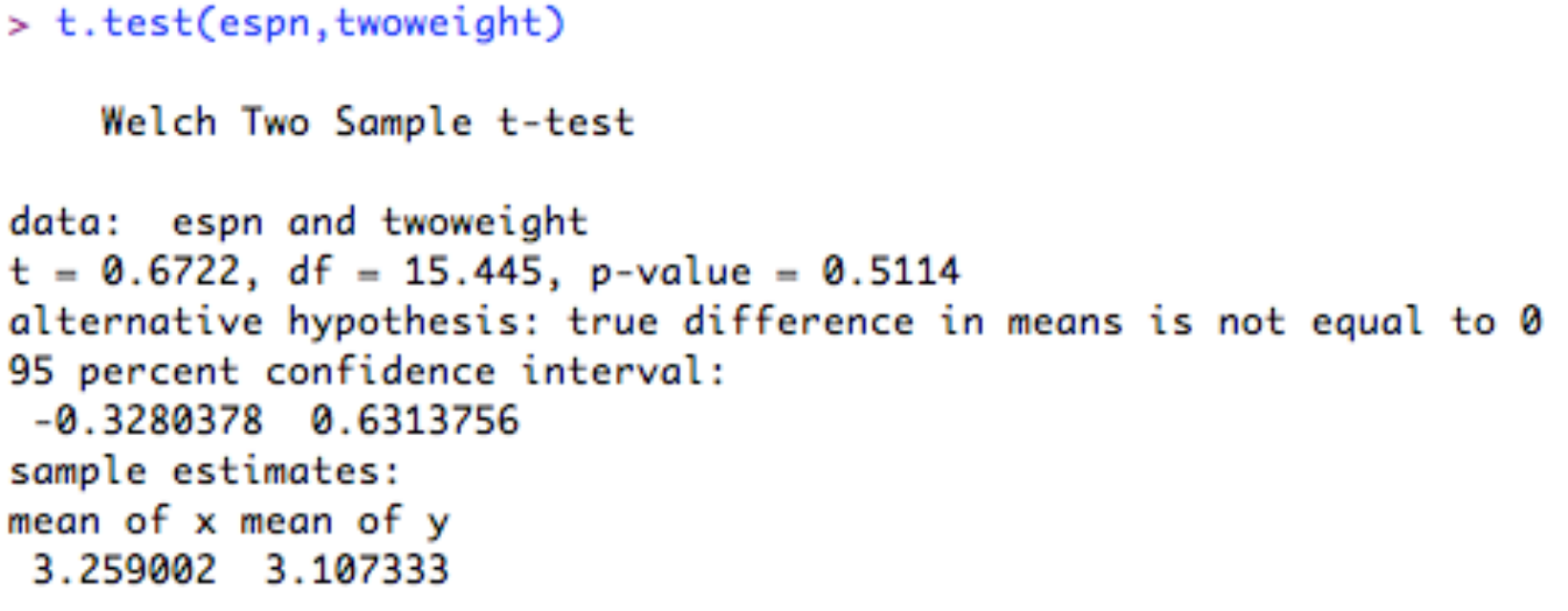}
\caption{t-test to determine whether the difference between the games off determined by ESPN ExWL and those by linear combinations of Weibulls is statistically significant.  \label{ttest2}}
\end{figure}

\begin{figure}[h!]
\includegraphics[scale=0.4]{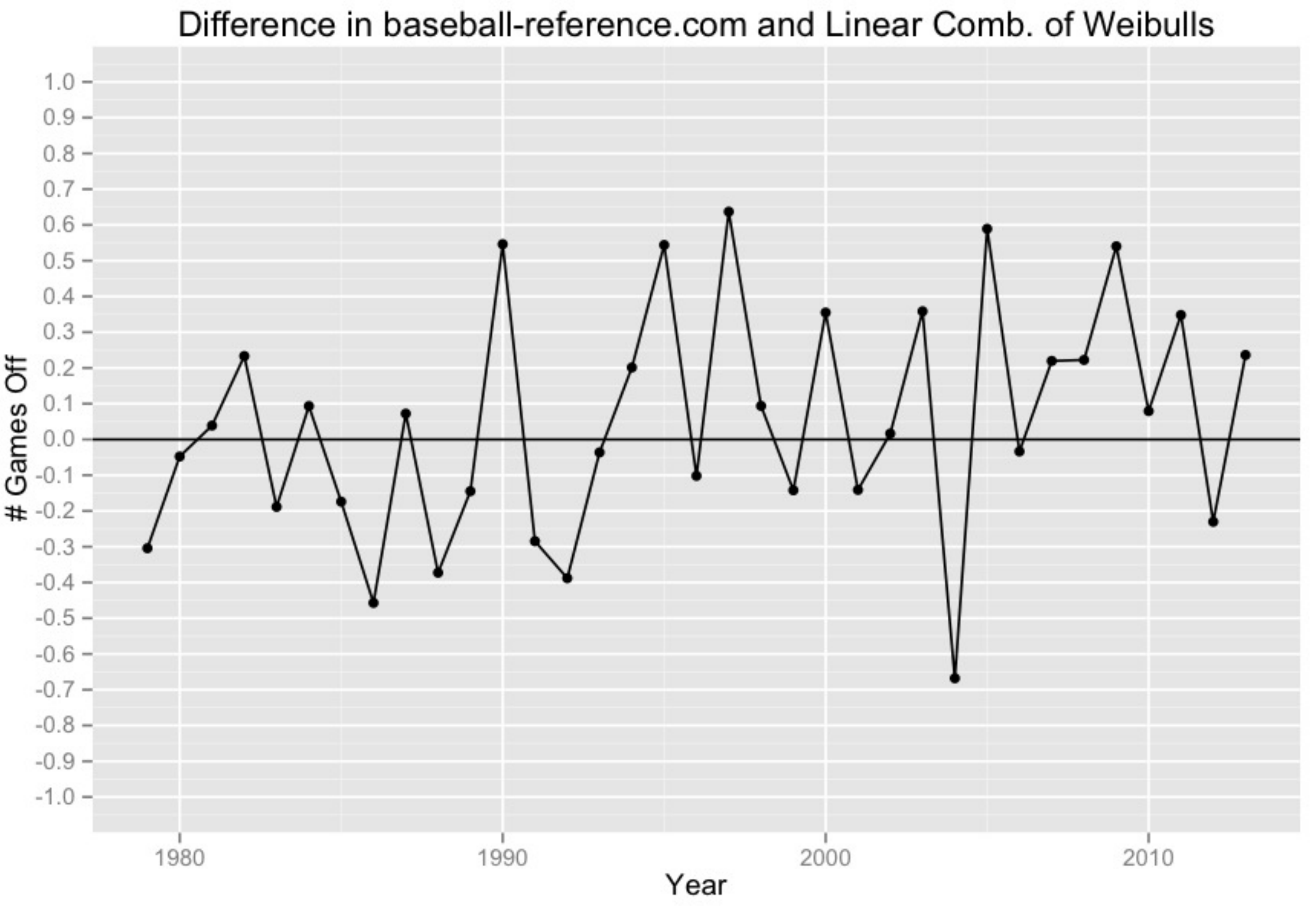}
\caption{Difference in mean number of games off for \url{baseball-reference.com}'s pythWL statistic and linear combination of Weibulls from 1979-2013. \label{difference}}
\end{figure}

We also performed $\chi^2$ tests to determine the goodness of fit to see how well the linear combination of Weibulls maps the observed data, and whether runs scored and allowed are independent. We used the bins as in \eqref{bins} and test statistic
\begin{align}
&\sum_{k=1}^{\textnormal{\# Bins}} \frac{(\rso(k)-\# \textnormal{Games} \ast  A(\alpha_{{\rm RS}_1},\alpha_{{\rm RS}_2}, -.5,\gamma,c_1,k))^2}{\# \textnormal{Games} \ast  A(\alpha_{{\rm RS}_1},\alpha_{{\rm RS}_2},-.5,\gamma,c_1,k)} \nonumber \\
&\indent + \sum_{k=1}^{\textnormal{\# Bins}} \frac{(\rao(k)-\# \textnormal{Games} \ast  A(\alpha_{{\rm RA}_1},\alpha_{{\rm RA}_2}, -.5,\gamma,c_1',k))^2}{\# \textnormal{Games} \ast  A(\alpha_{{\rm RA}_1},\alpha_{RS_A},-.5,\gamma,c_1',k)}
\end{align}
for the goodness of fit tests, with $2\ast (\# \textnormal{Bins} -1)-1-7 = 16$ degrees of freedom, the factor of 7 coming from estimating 7 parameters, namely $\alpha_{{\rm RS}_1}$, $\alpha_{{\rm RS}_2}$, $\alpha_{{\rm RA}_1}$, $\alpha_{{\rm RA}_2}$, $\gamma$, $c_1$, and $c_1'$. We did not estimate $\beta$, as we took it to be -.5. Having 16 degrees of freedom gives critical threshold values of 26.3 (at the 95\% level) and 32.0 (at the 99\% level). However, since there are multiple comparisons being done (namely 30 for the different teams), we use a Bonferroni adjustment and obtain critical thresholds of 37.7 (95\%) and 42.5 (99\%). From the first column of Figure \ref{independence}, all the teams fall within the unadjusted 99\% threshold, with the exception of the Texas Rangers (just barely!), who easily fall into the Bonferroni adjusted 95\% threshold. Therefore, the observed data closely follows a linear combination of Weibulls with the proper estimated parameters.

Since the test for independence of runs scored and allowed requires that the row and column of the contingency table have at least one non-zero entry, the bins used to bin the runs score and allowed were
\be
[0,1)\ \cup\ [1,2) \ \cup\ \cdots\ \cup\ [9,10)\ \cup\ [11,\infty).
\ee

We use integer endpoints because we are using the observed runs from games. We have a 12 by 12 contingency table with zeroes along the diagonal, since runs scored and allowed can never be equal. This leads to an incomplete 12 by 12 contingency table with $(12-1)^2-12=109$ degrees of freedom; constructing a test requires the use of structural zeroes. The theory behind tests using structural zeroes can be seen in \cite{Mil} or Appendix 9.2 of \cite{MCGLP}. We observe that 109 degrees of freedom give critical threshold values of 134.37 (at the 95\% level) and 146.26 (at the 99\% level). Again, since we are doing multiple comparisons, we use a Bonferroni adjustment, obtaining critical thresholds of 157.68 (95\%) and 166.45 (99\%). From the second column of Figure \ref {independence}, all the teams fall within the 99\% threshold, with the exception of the Los Angeles Angels (just barely!), who easily fall into the Bonferroni adjusted 95\% threshold. Thus, runs scored and allowed are acting as though they are statistically independent.

\begin{figure}[h!]
\includegraphics[scale=0.51]{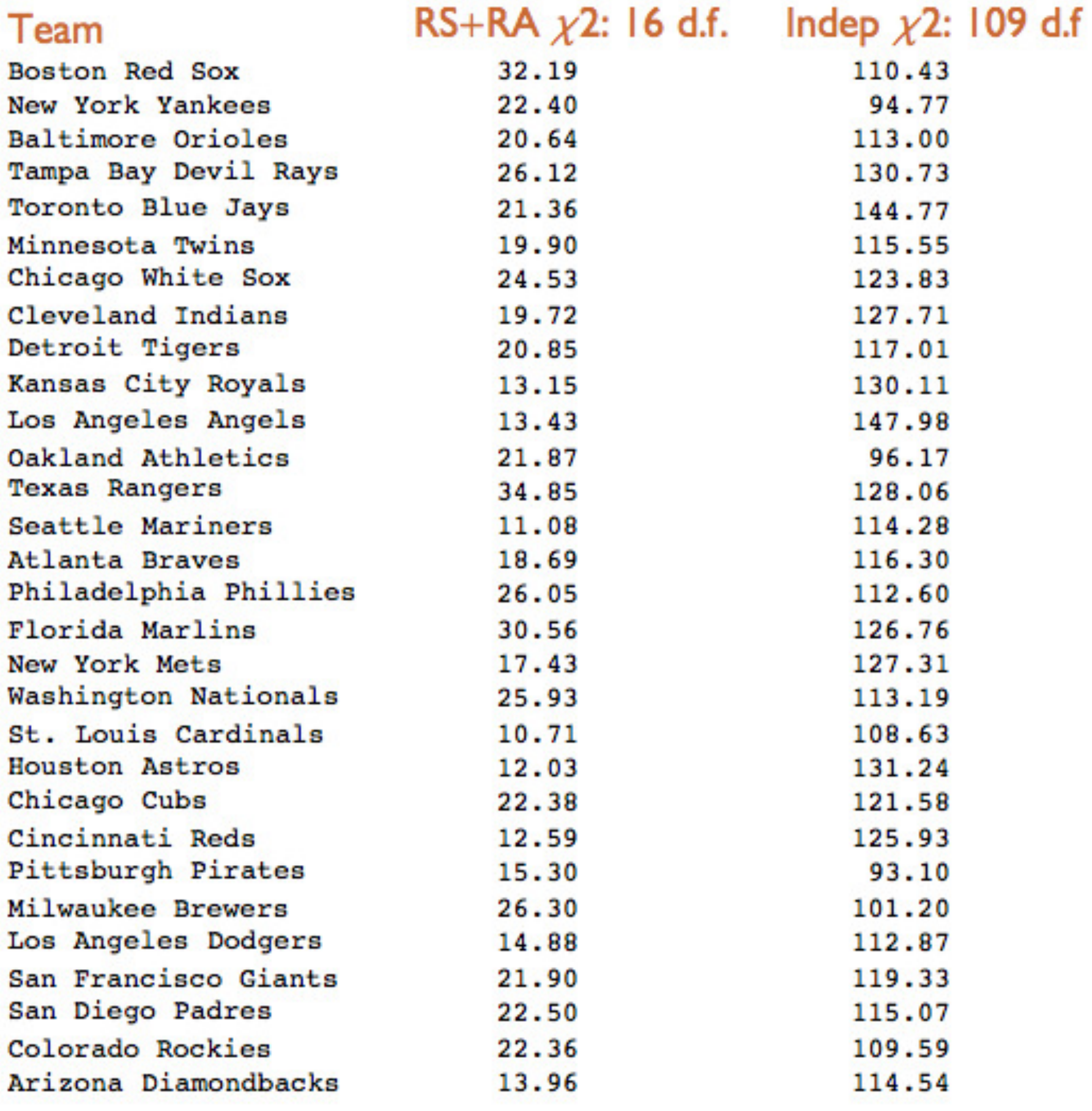}
\caption{$\chi^2$ test results of the 2011 season from least squares of goodness of fit and independence of runs score and allowed. \label{independence}}
\end{figure}

A more in depth discussion of the justification behind the tests can be found in \cite{Mil}.

\section{Future Work and Conclusions}

While a one game improvement in prediction is very promising, as our formula requires us to fit the runs scored and allowed distributions we explored simplifications. We tried to simplify the formula, even giving up some accuracy, in order to devise a formula that could be easily implemented using just a team's runs scored and allowed (and the variance of each of these) in order to determine the team's winning percentage. Unfortunately, the weight parameters $c_1$ and $c_1'$ plays too much of a factor; in 2011, the mean of the parameter $c_1$ is 0.21 with a standard deviation of 0.39 (and a median of 0.21). With such large fluctuations in the weight parameters from team to team, the task of finding a simpler formula was almost impossible, as creating a uniform formula that every team could use was not feasible when two of the key parameters were so volatile. Taking this into account, we tried fixing the $\gamma$, $c_1$, and $c_1'$ parameters, allowing for us to just solve a quartic involving the first and second moments to find the other parameters ($\alpha_{{\rm RS}_1}$, $\alpha_{{\rm RS}_2}$, $\alpha_{{\rm RA}_1}$, and $\alpha_{{\rm RA}_2}$). However, while we were able to solve for the other parameters, plugging these values of the parameters gave us a significantly worse prediction of teams' win-loss percentage compared to linear combination of Weibulls and \url{baseball-reference.com}'s Pythagorean Win-Loss statistic. One of the great attractions of James' Pythagorean formula is its ease of use; we hope to return to other simplifications and approximations in a later paper. Our hope is to find a linearization or approximation of our main result, similar to how Dayaratna and Miller \cite{DaMil1} showed the linear predictor of Jones and Tappin \cite{JT} follows from a linearization of the Pythagorean formula.

To summarize our results, using a linear combination of Weibulls rather than a single Weibull increases the prediction accuracy of a team's W/L percentage. More specifically, we saw that the single Weibull's predictions for a team's wins were on average 4.22 games off (with a standard deviation of 3.03), while the linear combination of Weibull's predictions for a team's wins were from 2004-2012 were on average 3.11 games off (with a standard deviation of 2.33), producing about a $25\%$ increase in prediction accuracy. We also performed $\chi^2$ goodness of fit tests for the linear combination of Weibulls and tested the statistical independence of runs scored and allowed (a necessary assumption), and see that in fact the linear combination of Weibulls with properly estimated parameters obtained from least squares analysis closely maps the observed runs and that runs scored and allowed are in fact statistically independent. In addition, when compared against \url{baseball-reference.com}'s Pythagorean Win-Loss statistic, the linear combination of Weibulls does .06 of a game better in the years from 1979 to 2013, but this improvement cannot be considered statistically significant. However, in more recent years, it is worth noting that it does appear that the linear combination of Weibulls is doing better than \url{baseball-reference.com}'s Pythagorean Win-Loss statistic.

\appendix

\section{Moments of Weibulls}\label{sec:AppMoment}

For the runs scored data, we have $c_2 = 1-c_1$ and the density equals
\begin{equation}
c_1\frac{\gamma}{\alpha_{{\rm RS}_1}} \left(\frac{x-\beta}{\alpha_{{\rm RS}_1}}\right)^{\gamma-1} e^{-(\frac{x-\beta}{\alpha_{{\rm RS}_1}})^{\gamma}}+(1-c_1)\frac{\gamma}{\alpha_{{\rm RS}_2}} \left(\frac{x-\beta}{\alpha_{{\rm RS}_2}}\right)^{\gamma-1} e^{-(\frac{x-\beta}{\alpha_{{\rm RS}_2}})^{\gamma}}.
\end{equation}
From \cite{Mur}, and using the fact that the two Weibulls in the linear combination are independent, we obtained the following moments:
\begin{align}
\textnormal{First Moment}  &\ =\ c_1(\alpha_{{\rm RS}_1}\Gamma(1+\gamma^{-1})+\beta)+(1-c_1)(\alpha_{{\rm RS}_2}\Gamma(1+\gamma^{-1})+\beta) \nonumber \\
\textnormal{Second Moment}  & \ =\  c_1^2\left[ \alpha_{{\rm RS}_1}^2\left( \Gamma(\frac{2}{\gamma}+1)-(\Gamma(\frac{1}{\gamma}+1))^2\right)\right] \nonumber \\ & \indent + (1-c_1)^2\left[ \alpha_{{\rm RS}_2}^2\left( \Gamma(\frac{2}{\gamma}+1)-(\Gamma(\frac{1}{\gamma}+1))^2\right)\right] \nonumber \\
\textnormal{Third Moment} & \ =\  (c_1^3+(1-c_1)^3)\ast \frac{g_3-3g_1g_2+2g_1^3}{(g_2-g_1^2)^{3/2}}\nonumber \\
\textnormal{Fourth Moment} & \ =\  (c_1^4+(1-c_1)^4)\ast \frac{g_4-4g_1g_3+6g_2g_1^2-3g_1^4}{(g_2-g_1^2)^2} \nonumber\\
& \indent + 6\ast c_1^2(1-c_1)^2\ast \left[\alpha_{{\rm RS}_1}^2\alpha_{{\rm RS}_2}^2 \left(g_2-g_1^2\right)^2 \right],
\end{align} where $g_i=\Gamma(1+\frac{i}{\gamma})$, and $\Gamma(x)$ is the Gamma function defined by $\Gamma(x)=\int_0^\infty e^{-u}u^{x-1}du$. We next used $R$ to find the observed moments for teams' runs scored from 2007. The code is available in \cite{Luo}.


\ \\

\end{document}